\documentclass[final,4p]{elsarticle}

\usepackage{amsmath,amsthm,amssymb,mathrsfs}
\usepackage[colorlinks,linkcolor=red, citecolor=green,filecolor=magenta,urlcolor=cyan]{hyperref}

\newtheorem{theorem}{Theorem}[section]
\newtheorem{corollary}[theorem]{Corollary}

\newtheorem{lemma}[theorem]{Lemma}

\newtheorem{definition}{Definition}[section]

\numberwithin{equation}{section}

\journal{}%journal%
\allowdisplaybreaks

\begin{document}

\title{Liouville Theorem for Harmonic Maps from Riemannian Manifold with Compact Boundary\tnoteref{SS}}

\author[whu1,whu2]{Jun Sun}
\ead{sunjun@whu.edu.cn}

\author[ruc]{Xiaobao Zhu}
\ead{zhuxiaobao@ruc.edu.cn}

\tnotetext[S]{This paper is supported by NSFC 12071352, NSFC 11721101, the National Key R and D Program of China 2020YFA0713100, and 2022CFB240.}

\address[whu1]{School of Mathematics and Statistics, Wuhan University, Wuhan 430072, China}
\address[whu2]{Hubei Key Laboratory of Computational Science, Wuhan University, Wuhan, 430072, China}
\address[ruc]{School of Mathematics, Renmin University of China, Beijing 100872, China}

%\cortext[ruc]{Corresponding author.}

\begin{abstract}
In this note we will provide a gradient estimate for harmonic maps from a complete noncompact Riemannian manifold with compact boundary (which we call ``Kasue manifold") into a simply connected complete Riemannian manifold with non-positive sectional curvature. As a consequence, we can obtain a Liouville theorem. We will also show the nonexistence of positive solutions to some linear elliptic equations on Kasue manifolds.
\end{abstract}


\maketitle

%main text%

\section{Introduction}

\allowdisplaybreaks

\noindent Gradient estimate is a very important technique in geometric analysis and has attracted much attentions since Yau's seminal paper (\cite{Yau}). Yau proved the following gradient estimate for positive harmonic functions:

\begin{theorem}(\cite{Yau})
Let $(M,g)$ be an $n$-dimensional, complete Riemmanian manifold (without boundary). For $K\geq 0$, we assume that $Ric_M\geq -(n-1)K$. For $x_0\in M$, let $u:B_R(x_0)\to (0,\infty)$ be a positive harmonic function. Then we have
\begin{align*}
    \frac{|\nabla u|}{u}\leq C_n\left(\frac{1}{R}+\sqrt{K}\right),
\end{align*}
on $B_{\frac{R}{2}}(x_0)$, where $C_n$ is a positive constant depending only on $n$.
\end{theorem}

A consequence of the gradient estimate is the well-known Liouville theorem, which states that any positive harmonic function on a complete Riemannian manifold with nonnegative Ricci curvature is a constant.

\vspace{.1in}

Later on, Cheng generalized Yau's gradient estimate to the harmonic map case (\cite{Cheng}). More precisely, he proved that:

\begin{theorem}(\cite{Cheng})
Let $M$ be a complete Riemannian manifold with Ricci curvature
bounded from below $Ric_M\geq -(n-1)K$. Let $N$ be a simply connected complete Riemannian manifold with non-positive sectional curvature. Let u be a harmonic map
from $M$ to $N$. Assume that $y_0\not\in u(B_R(0))$. Let $\rho(y)$ be the distance between $y$ and $y_0$ in $N$. Then, if $b>2\sup\{\rho(u(x))|x\in B_R(0)\}$, we have
\begin{align*}
    \sup_{B_{\frac{R}{2}}(0)}\frac{|\nabla u|(x)}{b^2-\rho^2(u(x))}\leq \frac{C(1+KR^2)}{R^2b},
\end{align*}
where $C>0$ depends only on $M$ and $N$.
\end{theorem}

As a consequence of Cheng's gradient estimate, we can obtain the Liouville theorem for harmonic maps:

\begin{corollary}
Let $M$ be a complete Riemannian manifold with nonnegative Ricci
curvature. Let $N$ be a simply connected complete Riemannian manifold with non-positive sectional curvature. Let $u$ be a harmonic map from $M$ to $N$. If the image of $u$ in $N$ is a bounded set, then it is constant.
\end{corollary}

\vspace{.1in}

Recently, Kunikawa and Sakurai (\cite{KS}) derived Yau's gradient estimate in the setting that the Riemmannian manifold $M$ is complete noncompact with compact boundary for harmonic functions with Dirichlet boundary condition (i.e., it is constant on the boundary). They proved that

\begin{theorem}(\cite{KS})
Let $(M,g)$ be an $n$-dimensional, complete Riemmanian manifold with compact boundary. For $K\geq 0$, we assume that $Ric_M\geq -(n-1)K$ and $H_{\partial M}\geq -(n-1)\sqrt{K}$. Let $u:B_R(\partial M)\to (0,\infty)$ be a positive harmonic function  with Dirichlet boundary condition. We assume that its derivative $u_{\nu}$ in the direction of the outward unit normal vector $\nu$ is non-negative over $\partial M$. Then we have
\begin{align*}
    \frac{|\nabla u|}{u}\leq C_n\left(\frac{1}{R}+\sqrt{K}\right),
\end{align*}
on $B_{\frac{R}{2}}(\partial M)$, where $C_n$ is a positive constant depending only on $n$, and $B_R(\partial M):=\{x\in M|d(x,\partial M)<R\}$.
\end{theorem}

In particular, they can obtain the following Liouville theorem:

\begin{corollary}\label{cor-KS}
Let $M$ be a complete Riemannian manifold with compact boundary. We assume that $Ric_M\geq 0$ and $H_{\partial M}\geq 0$. Let $u:M\to (0,\infty)$ be a positive harmonic function  with Dirichlet boundary condition. We assume that $u_{\nu}\geq 0$ over $\partial M$. Then $u$ is constant.
\end{corollary}

Actually, a Riemannian manifold $M$ satisfying the assumptions of Corollary \ref{cor-KS} is classified by Kasue (\cite{Kasue2}):

\begin{theorem}\label{thm-kasue} (\cite{Kasue2}) Let $M$ be a connected, complete noncompact Riemannian manifold with compact boundary. If $Ric_M\geq0$ and $H_{\partial M}\geq 0$, then $\partial M$ is connected, and $M$ is isometric to $[0,\infty)\times\partial M$.
\end{theorem}

\begin{definition}
We say that $M$ is a {\bf Kasue manifold} if it satisfies the assumptions in Theorem \ref{thm-kasue}.
\end{definition}

In this paper we will first prove a Liouville theorem for harmonic maps from a Kasue manifold to a simply connected Riemannian manifold with non-positive sectional curvature.

\vspace{.1in}

\noindent \textbf{Theorem A:} {\it
Let $M$ be a Kasue manifold, $N$ be a simply connected complete Riemannian manifold with non-positive sectional curvature and let $u$ be a harmonic map from $M$ to $N$ with Dirichlet boundary condition (i.e., $u$ is constant on the boundary). If the image of $u$ in $N$ is a bounded set, then it is constant.
}

\vspace{.1in}

The Liouville theorem will follow from the general gradient estimate (see Theorem \ref{thm-GE}). The idea of the proof of the gradient estimate follows from that of Cheng (\cite{Cheng}) which applied the maximum principle to appropriately chosen test function. The main difference is to deal with the case that the maximum of the test function is achieved on the boundary of $M$. For this reason, our gradient estimate holds with specially chosen $y_0$, while Cheng's gradient estimate holds for any chosen $y_0$. However, this is enough to guarantee the validity of the Liouville theorem.

\vspace{.1in}

The other aim of this paper is to show the nonexistence of positive solution to some elliptic equations on Kasue manifolds. Precisely,

\vspace{.1in}

\noindent \textbf{Theorem B:} {\it
Let $(M,g)$ be a Kasue manifold and suppose $h\in C^2(M)$ satisfies
\begin{align*}
\Delta h\geq0, ~~0\leq h\not\equiv0~~\text{and}~~h_\nu|_{\partial M}\leq0.
\end{align*}
Then the equation
\begin{align}\label{eq-linear}
\begin{cases}
&\Delta u+hu=0~~~~\text{in}~~~~M,\\
&u_\nu|_{\partial M}\geq0
\end{cases}
\end{align}
does not admit a positive solution.
}

\vspace{.1in}

To end the introduction, we want to say some words on the motivation of our results. When we read the article of Kunikawa and Sakurai \cite{KS}, we noticed that the Laplacian comparison theorem also holds for the distance function from the boundary. Relating to our knowledge on gradient estimates, we realized that we can derived some Liouville
properties for elliptic equations and harmonic maps. Then our results appeared. The method of our proof is standard. We believe that one can use it to study other elliptic equations
and parabolic equations on Kasue manifolds. Here, we refer the reader to \cite{FW} for gradient estimates for a nonlinear parabolic equation on Kasue manifolds.

\section{Proof of Theorem A}
In this section, we will prove Theorem A. We will use the classical idea to prove the gradient estimate (Theorem \ref{thm-GE}). To that purpose, we first recall the Laplacian comparison theorem for the distance function from a compact hypersurface in the Riemannian manifold $M$.

Let $(M,g)$ be an $n$-dimensional complete noncompact Riemannian manifold with compact boundary $\partial M$. The distance function from the boundary $r_{\partial M}:M\to {\mathbb R}$ is defined as
\begin{align*}
    r_{\partial M}:=d(\cdot,\partial M),
\end{align*}
which is smooth outside of the cut locus for the boundary Cut$\partial M$ (\cite{Sakurai}).

For $K,\Lambda\in {\mathbb R}$, we denote by $s_{K,\Lambda}(t)$ the unique solution to the Jacobi equation $\varphi''(t)+K\varphi(t)=0$ with initial conditions $\varphi(0)=1$ and $\varphi'(0)=-\Lambda$. We have the following Laplacian comparison theorem:

\begin{theorem}\label{thm-comparison}(\cite{Kasue1})
For $K,\Lambda\in {\mathbb R}$, we assume $Ric_M\geq (n-1)K$ and $H_{\partial M}\geq (n-1)\Lambda$. Then we have
\begin{align*}
    \Delta r_{\partial M}\leq (n-1)\frac{s'_{K,\Lambda}(r_{\partial M})}{s_{K,\Lambda}(r_{\partial M})}
\end{align*}
outside of Cut$\partial M$.
\end{theorem}

The following Bochner formula will be used in the derivation of gradient estimate:

\begin{lemma}\label{lemma-bochner}(\cite{ES})
Let $(M^n,g)$ and $(N^m,h)$ be two Riemannian manifolds and $u:M\to N$ be a harmonic map, then the energy density of $u$ satisfies the following formula:
\begin{equation*}
\frac{1}{2}\Delta|\nabla u|^2=|\nabla du|^2+\langle Ric^M\nabla u,\nabla u\rangle-\langle Rm^N(u_{\alpha},u_{\beta})u_{\alpha}, u_{\beta}\rangle,
\end{equation*}
where $Ric^M$ and $Rm^N$ are the Ricci curvature of $(M,g)$ and Riemannian curvature of $(N,h)$, respectively.
\end{lemma}

We will also need the following Reilly type formula which was first proved for functions by Reilly (\cite{Reilly}). For the map case, we refer to \cite{Matei}.

\begin{lemma}\label{lemma-reilly}(\cite{Matei})
Let $(M^n,g)$ be an $n$-dimensional Riemannian manifold with compact boundary $\partial M$, and $(N^m,h)$ be a Riemannian manifold. Let $\nu$ be the unit outer normal vector of $\partial M$ in $M$. Then for all smooth map $u:(M,g)\to (N,h)$, we have
\begin{align*}
\left(|\nabla u|^2\right)_{\nu}
=& 2\langle du(\nu),\tau^{\partial M}(u)-\tau (u)\rangle-2H|du(\nu)|^2+2\langle \nabla(du(\nu)),du\rangle_{\partial M}\\
&-2\langle du\circ A,du\rangle_{\partial M},
\end{align*}
where $Av:=\nabla^M_v \nu$ for $v\in T\partial M$ is the Weingarten operator, $H$ is the mean curvature of $\partial M$ in $M$ with respect to $\nu$, and $\tau^{\partial M}(u)$ and $\tau(u)$ are the tension fields of $u_{\partial M}:\partial M\to N$ and $u:M\to N$, respectively.
\end{lemma}

\vspace{.1in}

Now we can prove the gradient estimate.

\vspace{.1in}

\begin{theorem}\label{thm-GE}
Let $(M,g)$ be an $n$-dimensional, complete Riemmanian manifold with compact boundary and $N$ be a simply connected complete Riemannian manifold with non-positive sectional curvature. For $K\geq 0$, we assume that $Ric_M\geq -(n-1)K$ and $H_{\partial M}\geq -(n-1)\sqrt{K}$. Let $u:M\to N$ be a harmonic map with Dirichlet boundary condition. Assume that the image of $u$ is a bounded set in $N$. Then we can choose $y_0\not\in u(M)$ so that, if we let $\rho(y)$ be the distance between $y$ and $y_0$ in $N$, then for any $R>0$ we have
\begin{align*}
    \sup_{B_{\frac{R}{2}}(\partial M)}\frac{|\nabla u|(x)}{b^2-\rho^2(u(x))}\leq \frac{C(1+\sqrt{K}R)}{bR},
\end{align*}
for some constant $b>2\sup\{\rho(u(x))|x\in M\}$, where $C>0$ depends only on $M$ and $N$.
\end{theorem}

\begin{proof}
Let $(M^n,g)$ be an $n$-dimensional complete noncompact Riemannian manifold with compact boundary $\partial M$, and $(N^m,h)$ be a simply connected Riemannian manifold with non-positive sectional curvature. Let $u:M\to N$ be a smooth harmonic map with Dirichlet boundary condition. We assume that the image of $u$, denoted by $u(M)$, is  bounded in $N$, and $y_1:=u(\partial M)\in u(M)$. In the following, the constant $C$ will  denote a constant depending only on $n$, which may vary from line to line.

\vspace{.1in}

Now we fix a point $y_0\not\in u(M)$, which will be specified later. Let $\rho(y)$ be the distance between $y$ and $y_0$ in $N$. Fix $b>2\sup\{\rho(u(x))|x\in M\}$. As in \cite{Cheng}, we define
\begin{align*}
    \phi(x)=\frac{|\nabla u(x)|^2}{(b^2-\rho^2(u(x)))^2}.
\end{align*}
Then we have
\begin{align}\label{e-gradient}
    \nabla\phi(x)=\frac{\nabla(|\nabla u(x)|^2)}{(b^2-\rho^2(u(x)))^2}+2\frac{|\nabla u(x)|^2\nabla(\rho^2(u(x)))}{(b^2-\rho^2(u(x)))^3},
\end{align}
and
\begin{align}\label{e-laplace}
    \Delta\phi(x)=&\frac{\Delta(|\nabla u(x)|^2)}{(b^2-\rho^2(u(x)))^2}+4\frac{\nabla(|\nabla u(x)|^2)\cdot\nabla(\rho^2(u(x)))}{(b^2-\rho^2(u(x)))^3}\nonumber\\
    &+2\frac{|\nabla u(x)|^2\Delta(\rho^2(u(x)))}{(b^2-\rho^2(u(x)))^3}+6\frac{|\nabla u(x)|^2|\nabla(\rho^2(u(x)))|^2}{(b^2-\rho^2(u(x)))^4}.
\end{align}
The Bochner formula for harmonic maps together with our assumptions on curvatures implies that
\begin{align*}
    \Delta|\nabla u|^2\geq 2|\nabla du|^2-2(n-1)K|\nabla u|^2.
\end{align*}
By standard argument using Cauchy-Schwartz inequality and Hessian comparison theorem, we can obtain the following inequality:
\begin{align}\label{e-phi}
    \Delta\phi(x)\geq \frac{4|\nabla u(x)|^4}{(b^2-\rho^2(u(x)))^3}-\frac{2(n-1)K|\nabla u(x)|^2}{(b^2-\rho^2(u(x)))^2}+\frac{2\nabla\phi\cdot\nabla \rho^2}{b^2-\rho^2(u(x))}.
\end{align}

\vspace{.1in}

Now for any $R>0$, we define a function $F:B_R({\partial M})\to {\mathbb R}$ by
\begin{align*}
    F(x):=(R^2-r^2_{\partial M}(x))^2\phi(x).
\end{align*}
We assume that $F$ achieves its maximum at some point $x_0\in B_{R}(\partial M)$.

\vspace{.1in}

\textbf{Case 1: $x_0\in B_R(\partial M)\backslash\partial M$.} In this case, using Calabi's trick, we may assume that $x_0$ does not belong to Cut$\partial M$. Therefore, at $x_0$, it holds that
\begin{align*}
    \nabla F(x_0)=0
\end{align*}
and
\begin{align*}
    \Delta F(x_0)\leq0.
\end{align*}
Hence we have at $x_0$ that
\begin{align*}
    \frac{\nabla\phi}{\phi}=\frac{4r_{\partial M}\nabla r_{\partial M}}{R^2-r_{\partial M}^2}
\end{align*}
and
\begin{align*}
    \frac{\Delta\phi}{\phi}-\frac{8r_{\partial M}\nabla r_{\partial M}\cdot \nabla\phi}{(R^2-r_{\partial M}^2)\phi}-\frac{2\Delta r^2_{\partial M}}{R^2-r_{\partial M}^2}+\frac{8r^2_{\partial M}}{(R^2-r_{\partial M}^2)^2}\leq 0.
\end{align*}
It follows that
\begin{align*}
    \frac{\Delta\phi}{\phi}-\frac{24r^2_{\partial M}}{(R^2-r_{\partial M}^2)^2}-\frac{2\Delta r^2_{\partial M}}{R^2-r_{\partial M}^2}\leq 0.
\end{align*}
Notice that $s_{-K,-\sqrt{K}}(t)=e^{\sqrt{K}t}$. The Laplacian comparison theorem (Theorem \ref{thm-comparison}) and our assumptions imply that
\begin{align*}
   \Delta r^2_{\partial M}= 2r_{\partial M}\Delta r_{\partial M}+2\leq 2(n-1)\sqrt{K}r_{\partial M}+2\leq C(1+\sqrt{K}r_{\partial M}).
\end{align*}
Combining with (\ref{e-phi}), we obtain that
\begin{align*}
    0\geq &\frac{\Delta\phi}{\phi}-\frac{24r^2_{\partial M}}{(R^2-r_{\partial M}^2)^2}-\frac{2\Delta r^2_{\partial M}}{R^2-r_{\partial M}^2}\\
    \geq & 4(b^2-\rho^2)\phi-2(n-1)K+\frac{8r_{\partial M}\nabla r_{\partial M}\cdot \nabla\rho^2}{(R^2-r_{\partial M}^2)(b^2-\rho^2)}\\
    &-\frac{24r^2_{\partial M}}{(R^2-r_{\partial M}^2)^2}-\frac{C(1+\sqrt{K}r_{\partial M})}{R^2-r_{\partial M}^2}.
\end{align*}
Because
\begin{align*}
    \rho(u(x))\leq \frac{b}{2},
\end{align*}
and
\begin{align*}
    |\nabla(\rho^2(u(x)))|\leq 2\rho |\nabla^N\rho||\nabla u|\leq b|\nabla u|,
\end{align*}
we have
\begin{align*}
    3b^2\phi-\frac{8r_{\partial M}b|\nabla u|}{(R^2-r_{\partial M}^2)(b^2-\rho^2)}-\frac{24r^2_{\partial M}}{(R^2-r_{\partial M}^2)^2}-\frac{C(1+\sqrt{K}r_{\partial M})}{R^2-r_{\partial M}^2}-2(n-1)K\leq 0.
\end{align*}
Multiplying the above inequality through $(R^2-r_{\partial M}^2)^2$, we have
\begin{align*}
    3b^2F-8r_{\partial M}bF^{\frac{1}{2}}-C(1+\sqrt{K}r_{\partial M})R^2-2(n-1)KR^4\leq 0,
\end{align*}
which yields that
\begin{align*}
   \sup_{B_{\frac{R}{2}}(\partial M)}F^{\frac{1}{2}}(x)\leq F^{\frac{1}{2}}(x_0)\leq \frac{CR(1+\sqrt{K}R)}{b}.
\end{align*}
In particular, we have
\begin{align*}
   \sup_{B_{\frac{R}{2}}(\partial M)}\frac{|\nabla u|}{R^2-\rho^2}\leq F^{\frac{1}{2}}(x_0)\leq \frac{C(1+\sqrt{K}R)}{bR}.
\end{align*}

\vspace{.1in}

\textbf{Case 2: $x_0\in \partial M$.} In this case, we know that
\begin{align*}
    F_{\nu}(x_0)\geq0.
\end{align*}
Since $r_{\partial M}(x_0)=0$, we can easily see that
\begin{align}\label{e-nu}
   \phi_{\nu}(x_0)=\frac{F_{\nu}(x_0)}{R^4}\geq0.
\end{align}
Choose local orthonormal frame $\{e_1,\cdots,e_{n-1},\nu\}$ of $M$ along $\partial M$ around $x_0$ so that $T\partial M$ is spanned by $\{e_1,\cdots,e_{n-1}\}$. By our assumption, $u(\partial M)=y_1$. This implies that at $x_0$, we have
\begin{align*}
     |du(\nu)|=|\nabla u|, \ \ du(e_i)=0, \ for \ 1\leq i\leq n-1.
\end{align*}
Also we have $\tau(u)=0$ since $u$ is harmonic. Hence, Lemma 2.2 implies that
\begin{align*}
     (|\nabla u|^2)_{\nu}(x_0)=-2H(x_0)|\nabla u|^2(x_0).
\end{align*}
Therefore, (\ref{e-nu}) gives us at $x_0$ that
\begin{align}\label{e-inequality}
    0\leq& \phi_{\nu}=\frac{(|\nabla u|^2)_{\nu}}{(b^2-\rho^2)^2}+\frac{2|\nabla u|^2(\rho^2\circ u)_{\nu}}{(b^2-\rho^2)^3}\nonumber\\
    =&-2H\frac{|\nabla u|^2}{(b^2-\rho^2)^2}+\frac{4\rho|\nabla u|^2\langle \nabla^N\rho,du(\nu) \rangle}{(b^2-\rho^2)^3}.
\end{align}
Notice that $\nabla^N\rho(u(x_0))=\nabla^N\rho(y_1)$ is the radial direction from $y_0$ to $y_1$. Since $du(\nu)$ is independent of the choice of $y_0$, we can choose $y_0\not\in u(M)$ so that $\nabla \rho$ is in the direction of $-du(\nu)$. Since $u(M)$ is bounded in $N$, $y_0$ and $b$ can also be chosen so that
\begin{align*}
    \rho(y_1)\geq\inf\{\rho(u(x))|x\in M\}\geq \frac{2}{3}\sup\{\rho(u(x))|x\in M\}\geq \frac{1}{4}b.
\end{align*}
With this choice of $y_0$, we see that $\langle \nabla^N\rho,du(\nu) \rangle=-|du(\nu)|=-|\nabla u|$. Therefore, we have from (\ref{e-inequality}) that
\begin{align*}
b\phi^{\frac{3}{2}}(x_0)\leq -2H(x_0)\phi(x_0)\leq 2(n-1)\sqrt{K}\phi(x_0)
\end{align*}
so that
\begin{align*}
F^{\frac{1}{2}}(x_0)=R^2\phi^{\frac{1}{2}}(x_0)\leq \frac{C\sqrt{K}R^2}{b}\leq \frac{CR(1+\sqrt{K}R)}{b}.
\end{align*}
In particular, we have
\begin{align*}
   \sup_{B_{\frac{R}{2}}(\partial M)}\frac{|\nabla u|}{R^2-\rho^2}\leq F^{\frac{1}{2}}(x_0)\leq \frac{C(1+\sqrt{K}R)}{bR}.
\end{align*}
This finishes the proof of the theorem.
\end{proof}

\vspace{.1in}

\begin{proof}[Proof of Theorem A] The main theorem follows by setting $K=0$ and letting $R$ tend to infinity in Theorem \ref{thm-GE}.
\end{proof}

\section{Proof of Theorem B}

In this section, we shall give the proof of Theorem B, which is also based on a gradient estimate.

\begin{proof}[Proof of Theorem B]
We first assume that $(M,g)$ is an $n$-dimensional, complete Riemmanian manifold with compact boundary, $Ric_M\geq -(n-1)K$ and $H_{\partial M}\geq -(n-1)\sqrt{K}$ for some $K\geq 0$.

Suppose $u$ is a positive solution to equation
\begin{align*}
\Delta u+hu=0~~\text{on}~~M.
\end{align*}
Let
\begin{align*}
w=\frac{|\nabla u|^2}{u^2}.
\end{align*}
By choosing a local orthonormal system, we calculate the equation of $w$ as follows:
\begin{align}\label{eq3-1}
w_j=\frac{2u_iu_{ij}}{u^2}-\frac{2u_i^2u_j}{u^3}
\end{align}
and
\begin{align*}
\Delta w=w_{jj}=\frac{2u_{ij}^2}{u^2}+\frac{2u_iu_{ijj}}{u^2}-\frac{4u_iu_{ij}u_j}{u^3}-\frac{4u_iu_{ij}u_j}{u^3}-\frac{2u_i^2u_{jj}}{u^3}+\frac{6u_i^2u_j^2}{u^4}.
\end{align*}
Then by the Ricci formula and (\ref{eq3-1}), we have
\begin{align*}
\Delta w
=&2\left(\frac{u_{ij}}{u}-\frac{u_iu_j}{u^2}\right)^2
         -2\nabla w\cdot\nabla\log u+\frac{2u_i(\Delta u)_i}{u^2}+\frac{2R_{ij}u_iu_j}{u^2}-\frac{2u_i^2\Delta u}{u^3}\\
\geq&\frac{2}{n}\left(\frac{\Delta u}{u}-\frac{|\nabla u|^2}{u^2}\right)^2-2\nabla w\cdot\nabla\log u-2\nabla h\cdot\nabla\log u-2(n-1)Kw\\
=&\frac{2}{n}(w+h)^2-2\nabla(w+h)\cdot\nabla\log u-2(n-1)Kw.
\end{align*}
Since we assume $\Delta h\geq0$, one obtains
\begin{align}\label{eq3.2}
\Delta(w+h)\geq\frac{2}{n}(w+h)^2-2\nabla(w+h)\cdot\nabla\log u-2(n-1)Kw.
\end{align}

Let $\psi$ be a smooth cut-off function supported in $B_{R}(\partial M)$, satisfying the following properties:

(1) $\psi(x)=\psi(r_{\partial M}(x))$ and $\psi'\leq0$;

(2) $0\leq\psi\leq1$ and $\psi\equiv1$ in $B_{R/2}(\partial M)$;

(3) $\frac{|\psi'|}{\psi^a}\leq\frac{C_a}{R}$, $\frac{|\psi''|}{\psi^a}\leq\frac{C_a}{R^2}$ when $a\in(0,1)$.

\textbf{Case 1:} If $(w+h)\psi$ attains its maximum at some point $x_1\in B_{R}(\partial M)\setminus\partial M$, by using Calabi's argument we can assume w.l.o.g. that $x_1\not\in \text{Cut} \partial M$, then
$$\nabla[(w+h)\psi](x_1)=0, ~~~~~~\Delta[(w+h)\psi](x_1)\leq0.$$
Calculating directly and using (\ref{eq3.2}), one has at $x_1$ that
\begin{align}\label{eq3.3}
0\geq&\psi\Delta(w+h)+2\nabla(w+h)\cdot\nabla\psi+(w+h)\Delta\psi\nonumber\\
\geq&\frac{2}{n}(w+h)^2\psi+2(w+h)\nabla\psi\cdot\nabla\log u-2(n-1)Kw\psi\nonumber\\
&-2(w+h)\frac{|\nabla\psi|^2}{\psi}+(w+h)\Delta\psi.
\end{align}
We multiply inequality (\ref{eq3.3}) by $\psi(x_1)$ and estimate the new right-hand-side term by term.

Firstly, since $h$ is nonnegative, it follows from Cauchy's inequality and the properties of $\psi$ that
\begin{align}\label{eq3.3-1}
2(w+h)\psi\nabla\psi\cdot\nabla\log u\geq&-(w+h)^{3/2}\psi|\psi'|\nonumber\\
\geq&-\frac{1}{4n}(w+h)^2\psi^2-C(n)\left(\frac{|\psi'|}{\psi^{1/2}}\right)^{4}\nonumber\\
\geq&-\frac{1}{4n}(w+h)^2\psi^2-C(n)\frac{1}{R^4}.
\end{align}
Secondly, we have
\begin{align}\label{eq3.3-2}
-2(n-1)Kw\psi^2\geq-\frac{1}{4n}(w+h)^2\psi^2-C(n)K^2.
\end{align}
Thirdly, one has
\begin{align}\label{eq3.3-3}
-2(w+h)|\nabla\psi|^2\geq&-\frac{1}{4n}(w+h)^2\psi^2-C(n)\left(\frac{|\psi'|^2}{\psi}\right)^2\nonumber\\
\geq&-\frac{1}{4n}(w+h)^2\psi^2-C(n)\frac{1}{R^4}.
\end{align}
Lastly, by the Laplace comparison theorem (Theorem \ref{thm-comparison}) we have
\begin{align}\label{eq3.3-4}
(w+h)\psi\Delta\psi\geq&-\frac{1}{4n}(w+h)^2\psi^2-C(n)(\Delta\psi)^2\nonumber\\
=&-\frac{1}{4n}(w+h)^2\psi^2-C(n)[\psi''+\psi'\Delta r_{\partial M}]^2\nonumber\\
\geq&-\frac{1}{4n}(w+h)^2\psi^2-C(n)[|\psi''|+|\psi'|\frac{n-1}{r_{\partial M}}(1+\sqrt{K})r_{\partial M}]^2\nonumber\\
\geq&-\frac{1}{4n}(w+h)^2\psi^2-C(n)\left(\frac{1}{R^4}+K^2\right).
\end{align}
Substituting (\ref{eq3.3-1})-(\ref{eq3.3-4}) into (\ref{eq3.3}), one obtains
\begin{align*}
\frac{1}{n}[(w+h)\psi]^2(x_1)\leq C(n)\left(\frac{1}{R^4}+K^2\right).
\end{align*}
Therefore,
\begin{align*}
[(w+h)\psi](x)\leq[(w+h)\psi](x_1)\leq C(n)\left(\frac{1}{R^2}+K\right).
\end{align*}
Since $\psi\equiv1$ in $B_{R/2}(\partial M)$, we have
\begin{align*}
\sup_{B_{R/2}(\partial M)}w\leq\sup_{B_{R/2}(\partial M)}(w+h)\leq C(n)\left(\frac{1}{R^2}+K\right).
\end{align*}

\textbf{Case 2:} If $(w+h)\psi$ attains its maximum at some point $x_2\in\partial M$, one has at $x_2$ that
$$[(w+h)\psi]_\nu\geq0,$$
and hence
$$w_\nu\geq(w+h)_\nu=\frac{1}{\psi}\left([(w+h)\psi]_\nu-(w+h)\psi_\nu\right)=[(w+h)\psi]_\nu\geq0.$$
The Dirichlet boundary condition for $u$ and the assumption $u_\nu\geq0$ tell us that $|\nabla u|=u_\nu$ on $\partial M$.
It follows by the Reilly formula (see \cite{Reilly} or Proposition 2.3 in \cite{KS}) that
\begin{align*}
0\leq& w_\nu \leq(|\nabla\log u|^2)_{\nu}=2(\log u)_{\nu}(\Delta u-(\log u)_\nu H)\\
=&2\frac{u_\nu}{u}\left(\frac{\Delta u}{u}-\frac{|\nabla u|^2}{u^2}-\frac{u_\nu}{u}H\right)\\
=&2\frac{u_\nu}{u}\left(-h-\frac{|\nabla u|^2}{u^2}-\frac{u_\nu}{u}H\right)\\
\leq&2\frac{u_\nu}{u}\left(-w-\frac{u_\nu}{u}H\right).
\end{align*}
Therefore,
\begin{align*}
w\leq\max\{0,-w^{1/2}H\}\leq(n-1)\sqrt{K}w^{1/2}.
\end{align*}
We then have
$$w\leq C(n)K.$$

Combining the above two cases, we obtain finally that
\begin{align}\label{eq-final}
\sup_{B_{R/2}(\partial M)}\frac{|\nabla u|^2}{u^2}=\sup_{B_{R/2}(\partial M)}w\leq C(n)\left(\frac{1}{R^2}+K\right).
\end{align}

When $M$ is a Kasue manifold, we have $K=0$. Then by passing $R$ tend to infinity in (\ref{eq-final}), one has $u$ is a constant. Recalling that $h$ is nonnegative and not identical to $0$, the equation (\ref{eq-linear}) can not admit a positive constant solution.
So we arrive at that (\ref{eq-linear}) has no positive solution and complete the proof of Theorem B.
\end{proof}

\vspace{0.4in}

%bibtex
%\bibliographystyle{elsarticle-harv}%amsplain,abbrv,elsarticle-num,elsarticle-harv,elsarticle-num-names
%\bibliography{sj_bibtex}

\end{document}